\documentclass[11pt,twoside]{amsart}

\usepackage[utf8]{inputenc}
\usepackage[english]{babel}
\usepackage[T1]{fontenc}
\usepackage{amsfonts}
\usepackage{geometry}
\usepackage{pstricks, pstricks-add, pst-node, pst-tree}
\usepackage{amsmath, amsthm, amssymb, mathtools}
\usepackage{latexsym}
\usepackage{enumitem}
\usepackage{fancyhdr}
\usepackage{bm}
\usepackage[bookmarks=true]{hyperref}

\usepackage{libertine}

\usepackage[scaled=.8]{beramono}

\theoremstyle{plain}
\newtheorem{thm}{Theorem}[section]
\newtheorem{prop}[thm]{Proposition}
\newtheorem{cor}[thm]{Corollary}
\newtheorem{lem}[thm]{Lemma}
\theoremstyle{definition}

\newtheorem{exa}[thm]{Example}

\newtheorem{Notation}[thm]{Notation}
\newtheorem{question}{Question}
\newtheorem*{question*}{Question}

\newtheorem{rem}[thm]{Remark}

\date{}

\renewcommand{\arraystretch}{1.5}

\newcommand{\hgt}{{\rm ht}}

\newcommand{\pd}{{\rm pd}}
\newcommand{\ara}{{\rm ara}}
\newcommand{\cd}{{\rm cd}}
\newcommand{\chara}{{\rm char}}

\newcommand{\Z}{\mathbb{Z}}

\title{Cohomological dimension and arithmetical rank\\ of some determinantal ideals}

\author{Davide Bolognini, Alessio Caminata, Antonio Macchia, Maral Mostafazadehfard}

\address{{\small Davide Bolognini, Dipartimento di Matematica, Università di Genova, Via Dodecaneso 35, 16146 Genova, Italy}}
\email{{\small bolognin@dima.unige.it}}

\address{{\small Alessio Caminata, Institut f\"ur Mathematik, Universit\"at Osnabr\"uck, Albrechtstrasse 28a, 49076 Osnabr\"uck, Germany}}
\email{{\small alessio.caminata@uni-osnabrueck.de}}

\address{{\small Antonio Macchia, Fachbereich Mathematik und Informatik, Philipps-Universit\"at Marburg, Hans-Meerwein-Strasse 6, 35032 Marburg, Germany}}
\email{{\small macchia.antonello@gmail.com}}

\address{{\small Maral Mostafazadehfard, Departamento de Matem\'atica, CCEN Universidade Federal de Pernambuco, 50740-560 Recife, PE, Brazil}}
\email{{\small maralmostafazadehfard@gmail.com}}

\begin{document}
\thanks{The third author was supported by Università degli Studi di Bari.
\\ \indent The fourth author was partially supported by IMPA (Instituto Nacional de Matem\'atica Pura e Aplicada), Rio de Janeiro, Brazil}

\begin{abstract}
Let $M$ be a $(2 \times n)$ non-generic matrix of linear forms in a polynomial ring. For large classes of such matrices, we compute the cohomological dimension (cd) and the arithmetical rank (ara) of the ideal $I_2(M)$ generated by the $2$-minors of $M$. Over an algebraically closed field, any $(2 \times n)$-matrix of linear forms can be written in the Kronecker-Weierstrass normal form, as a concatenation of scroll, Jordan and nilpotent blocks. B\u{a}descu and Valla computed $\ara(I_2(M))$ when $M$ is a concatenation of scroll blocks. In this case we compute $\cd(I_2(M))$ and extend these results to concatenations of Jordan blocks. Eventually we compute $\ara(I_2(M))$ and $\cd(I_2(M))$ in an interesting mixed case, when $M$ contains both Jordan and scroll blocks. In all cases we show that $\ara(I_2(M))$ is less than the arithmetical rank of the determinantal ideal of a generic matrix. %Moreover, we show that often we can ignore the nilpotent blocks.
\end{abstract}

%%%%
\maketitle
%%%%


\section*{Introduction}

Determinantal ideals are a classical topic in Commutative Algebra and have been extensively studied because of their connections with other fields, such as Algebraic Geometry, Combinatorics, Invariant Theory and Representation Theory (see e.g. \cite{BV88}). In this paper we focus on the ideals $I_2(M)$ generated by the $2$-minors of a $(2 \times n)$ non-generic matrix $M$ in a polynomial ring $R$ over a field $K$. In particular, we compute the cohomological dimension ($\cd$) and the arithmetical rank ($\ara$) for large classes of such matrices.

We recall that the \textit{cohomological dimension} of an ideal $I$ of a Noetherian ring $R$ is
\[
\cd_R(I)= \max\{i \in \Z : H^i_I(R) \neq 0 \},
\]
where $H^i_I(R)$ denotes the $i$-th local cohomology module of $R$ with support in $I$, and the \textit{arithmetical rank} of $I$ is the smallest integer $s$ for which there exist $s$ elements of $R$, $a_1,\dots,a_s$, such that $\sqrt{I}=\sqrt{(a_1,\dots,a_s)}$. If there is no ambiguity, we will write simply $\cd(I)$ and omit the subscript $R$. In general, the following inequalities hold (see, e.g., \cite[Proposition 9.2]{ILLMMSW07}):
\[
\hgt(I) \leq \cd(I) \leq \ara(I),
\]
where $\hgt$ is the height of the ideal. If $\hgt(I) = \ara(I)$, then $I$ is called a \textit{set-theoretic complete intersection}. In particular, if $I$ is a squarefree monomial ideal in the polynomial ring $R=K[x_1,\dots,x_n]$, then
\begin{equation} \label{inequalities}
\hgt(I) \leq \pd_R(R/I) = \cd(I) \leq \ara(I) \leq \mu(I),
\end{equation}
where $\mu(I)$ denotes the minimum number of generators of $I$ and the equality between the projective dimension (pd) and the cohomological dimension was proved by Lyubeznik in \cite[Theorem 1]{L83}.

For a generic $(2 \times n)$-matrix $X$, Bruns and Schw\"anzl have shown in \cite{BS90} that $\ara(I_2(X))=2n-3$ and it is independent of the field. On the other hand, the cohomological dimension has a different behavior:
\[
\cd(I_2(X))=
\begin{cases}
\hgt(I_2(X)) = n-1 & \text{if } \chara(K)=p>0 \\
\ara(I_2(X)) = 2n-3 & \text{if } \chara(K)=0
\end{cases}.
\]

Motivated by \cite[Question 8.1]{LSW14}, we investigate the following special case. %we ask if $I_2(M)$ can be generated up to radical by less than $\ara(I_2(X))=2n-3$ elements, where $M$ is a $(2 \times n)$ non-generic matrix. In this paper,

\begin{question} \label{question}
Let $M=(x_{ij})$ be a $(2 \times n)$ non-generic matrix of linear forms and consider the ideal $I_2(M)$ in the polynomial ring $R=K[x_{ij}]$ generated by the $2$-minors of $M$. If $X$ is a $(2 \times n)$-generic matrix, is it true that $I_2(M)$ can be generated up to radical by less than $\ara(I_2(X))=2n-3$ elements, i.e. $\ara(I_2(M)) < \ara(I_2(X))$?
\end{question}

In order to study non-generic matrices, we first introduce the Kronecker-Weierstrass normal form of a matrix: a $(2 \times n)$-matrix $M$, whose entries are linear forms, can be written, by means of an invertible transformation, as a concatenation of blocks. Each block can be a nilpotent, a scroll or a Jordan matrix (see Section \ref{sectionKW}). First we treat all the possible mixed cases of $(2 \times 3)$-matrices in Remark \ref{2x3case}. In Example \ref{purecase} we compute cd and ara when $M$ consists of exactly one block. In all three cases $I_2(M)$ is a set-theoretic complete intersection with $\ara(I_2(M))=n-1$. In the rest of the paper we deal with matrices consisting of at least $2$ blocks and with $n \geq 4$ columns. In Proposition \ref{nilpotentiprop} we show that, if $X$ is a matrix of linear forms and we add a nilpotent block $N_n$ with length $n+1$ defining a new matrix $M=(X|N_n)$, then $\cd(I_2(M)) = \cd(I_2(X))+n$ and $\ara(I_2(M)) \leq \ara(I_2(X))+n$. This implies that, if we have a matrix $X$ for which $\cd(I_2(X))=\ara(I_2(X))$, then the concatenation of an arbitrary number of nilpotent blocks to $X$ preserves the equality between cohomological dimension and arithmetical rank.

In all the cases examined throughout the paper, we noticed a behavior similar to the generic case: the upper bound for the arithmetical rank is independent of the field, while the cohomological dimension is equal to the height of the ideal in positive characteristic and to the arithmetical rank in characteristic zero.

In Section \ref{sectionscrolls} we analyze concatenations of scroll blocks. B\u{a}descu and Valla, in \cite{BV10}, computed the arithmetical rank of the ideal $I_2(M)$, showing that it is independent of the field. On the other hand, using some tools from Algebraic Geometry, we prove that the cohomological dimension equals the height of the ideal if $\chara(K)=p>0$, while it is equal to the arithmetical rank if $\chara(K)=0$ (see Theorem \ref{scrolltheorem}).

In Section \ref{sectionjordan} we consider concatenations of Jordan blocks when $\chara(K)=0$. We show that also in this situation $\cd(I_2(M))=\ara(I_2(M))$.

Finally, in Section \ref{section2zeros}, we study an interesting mixed case. We start with a $(2 \times n)$-matrix $M$ with $2$ zeros in different rows and columns, and we transform it in the Kronecker-Weierstrass form. In this way $M$ can be written as a concatenation of two Jordan blocks of length $1$ with different eigenvalues and $n-2$ scroll blocks of length $1$. The ideal $I_2(M)$ is generated by both monomials and binomials. First we find an upper bound for the arithmetical rank independent of the field, showing that $\ara(I_2(M)) \leq 2n-5$. In the proof of Theorem \ref{ara2zeros} we combine the classical result by Bruns and Schw\"anzl (Theorem \ref{BrunsPoset}) and a well-known technique due to Schmitt and Vogel (Lemma \ref{SchmittVogel}). To reduce the number of generators up to radical, we sum some of them in a suitable way and use Pl\"ucker relations to prove the claim. Concerning the cohomological dimension, for small values of $n$, the ideal $I_2(M)$ is a set-theoretic complete intersection. For $n \geq 5$, in Theorem \ref{AraCd2zeros} we prove that $\cd(I_2(M))=\hgt(I_2(M))$ if $\chara(K)=p>0$, while $\cd(I_2(M))=\ara(I_2(M))$ if $\chara(K)=0$. For the last fact, we prove a stronger result, showing also the vanishing of all local cohomology modules with indices between the height and $2n-5$, if $\chara(K)=0$.

For all the classes of $(2 \times n)$-matrices considered in Sections \ref{sectionscrolls}, \ref{sectionjordan} and \ref{section2zeros}, except for small values of $n$, we always prove that $I_2(M)$ can be generated with less than $2n-3$ polynomials up to radical. Hence we give a positive answer to Question \ref{question}.

\section{Preliminaries}
In this section we recall some results that will be useful in the rest of the paper.

A well-known technique that provides an upper bound for the arithmetical rank of an arbitrary ideal is due to Schmitt and Vogel.

\begin{lem} \bf (Schmitt, Vogel \cite[Lemma p. 249]{SV79}) \it \label{SchmittVogel}
Let $R$ be a ring, $P$ be a finite subset of elements of $R$ and $P_0,\dots,P_r$ subsets of\, $P$ such that
\begin{itemize}
  \item[$(i)$] $\bigcup_{\ell=0}^r P_\ell = P$,
  \item[$(ii)$] $P_0$ has exactly one element,
  \item[$(iii)$] if $p$ and $p''$ are different elements of $P_\ell$, with $0 \leq \ell \leq r$, there is an integer $\ell'$, with $0 \leq \ell' < \ell$, and an element $p' \in P_{\ell'}$ such that $pp'' \in (p')$.
\end{itemize}
We set $\displaystyle{q_\ell = \sum_{p \in P_\ell} p^{e(p)}}$, where $e(p) \geq 1$ are arbitrary integers. We will write $(P)$ for the ideal of $R$ generated by the elements of $P$. Then $\sqrt{(P)} = \sqrt{(q_0,\dots,q_r)}$.
\end{lem}

In \cite{B89} and \cite{BS90}, Bruns and Schw\"anzl computed the cohomological dimension and the arithmetical rank of determinantal ideals of generic matrices. Let $X$ be an $(m \times n)$-matrix of indeterminates and $I_t(X)$ be the ideal generated by the $t$-minors of $X$.

\begin{thm} \bf (Bruns, Schw\"anzl, \cite[Theorem 2]{BS90}) \it \label{BrunsPoset}
Let $X$ be an $(m \times n)$-matrix of indeterminates over a ring $R$. Then
\[
\ara(I_t(X)) = mn - t^2 + 1.
\]
\end{thm}

In \cite[Corollary 2.2]{B89}, Bruns proved that $\ara(I_t(X)) \leq mn-t^2-1$ over any commutative ring, by defining a poset attached to the matrix $X$. We recall here this construction. We denote by $[a_1,\dots,a_t | b_1,\dots,b_t]$ the minor of $X$ with row indices $a_1,\dots,a_t$ and column indices $b_1,\dots,b_t$. On the set $\Delta(X)$ of all minors of $X$ we define a partial order given by
\begin{equation} \label{PartialOrder}
[a_1,\dots,a_u | b_1,\dots,b_u] \leq [c_1,\dots,c_v | d_1,\dots,d_v] \iff u \geq v, a_i \leq c_i \text{ and } b_i \leq d_i, i=1,\dots,v.
\end{equation}

The polynomials that generate $I_t(X)$ up to radical have the form
\begin{equation} \label{BrunsPolynomials}
p_j = \sum_{\substack{\xi \in \Delta(X),\\ \mathrm{rk}(\xi) = j}} \xi^{e(\xi)}, \text{ for } j=1,\dots,\mathrm{rk}(\Delta(X)),
\end{equation}
where $\mathrm{rk}(\Delta(X))$ denotes the rank of the poset, $e(\xi) = \frac{m}{\deg \xi}$ and $m$ is the least common multiple of the degrees of the elements $\xi \in \Delta(X)$.
\\ In particular, we are interested in the case $t=m=2$, for which
\[
p_j = \sum_{k=0}^{\lfloor \frac{j+1}{2} \rfloor -1 - \delta_j} [k+1+\delta_j,j-k+1-\delta_j]
\]
for $j=1,\dots,2n-3$, where $\delta_j= (j-n+1) \lfloor \frac{j}{n} \rfloor$. Here and in what follows, when we deal with $2$-minors, we use the notation $[a,b]$ instead of $[a,b|1,2]$.

\begin{exa}
We give an explicit example of the construction of the poset and of the polynomial generators up to radical for the ideal $I_2(X)$, where
\[
X=\begin{pmatrix}
x_1 & x_2 & x_3 & x_4 & x_5 \\
x_6 & x_7 & x_8 & x_9 & x_{10}
\end{pmatrix}.
\]
The poset $\Delta(X)$ is
\begin{figure}[ht!]
\psset{yunit=.7cm}
\begin{pspicture}(-1,-4.5)(12.4,1.8)
\psline(0,0)(6,0)
\psline(4,-1.5)(8,-1.5)
\psline(8,-3)(10,-3)
\psline(2,0)(4,-1.5)
\psline(4,0)(8,-3)
\psline(6,0)(12,-4.5)
\psdots(0,0)
\psdots(2,0)
\psdots(4,0)
\psdots(6,0)
\psdots(4,-1.5)
\psdots(6,-1.5)
\psdots(8,-1.5)
\psdots(8,-3)
\psdots(10,-3)
\psdots(12,-4.5)
\uput[70](0,0){$[1,2]$}
\uput[70](2,0){$[1,3]$}
\uput[70](4,0){$[1,4]$}
\uput[70](6,0){$[1,5]$}
\uput[70](4,-1.5){$[2,3]$}
\uput[70](6,-1.5){$[2,4]$}
\uput[70](8,-1.5){$[2,5]$}
\uput[70](8,-3){$[3,4]$}
\uput[70](10,-3){$[3,5]$}
\uput[70](12,-4.5){$[4,5]$}

\rput(-1,1.7){$\mathrm{rk}$}
\rput(0,1.7){$1$}
\rput(2,1.7){$2$}
\rput(4,1.7){$3$}
\rput(6,1.7){$4$}
\rput(8,1.7){$5$}
\rput(10,1.7){$6$}
\rput(12,1.7){$7$}
\end{pspicture}
\caption{} \label{poset}
\end{figure}

The ideal $I_2(X)$ is generated by the following $7$ polynomials up to radical:
\begin{align*}
p_1 &= [1,2] = x_1x_7-x_2x_6,\\
p_2 &= [1,3] = x_1x_8-x_3x_6,\\
p_3 &= [1,4] + [2,3] = x_1x_9-x_4x_6 + x_2x_8-x_3x_7,\\
p_4 &= [1,5] + [2,4] = x_1x_{10}-x_5x_6 + x_2x_9-x_4x_7,\\
p_5 &= [2,5] + [3,4] = x_2x_{10}-x_5x_7 + x_3x_9-x_4x_8,\\
p_6 &= [3,5] = x_3x_{10}-x_5x_8,\\
p_7 &= [4,5] = x_4x_{10}-x_5x_9.
\end{align*}
\end{exa}

While the arithmetical rank of $I_t(X)$ is independent of the ring, the cohomological dimension has a different behavior. In fact, if $R$ is a polynomial ring on a field of characteristic $0$, then $\cd(I_t(X)) = \ara(I_t(X)) = mn-t^2+1$ (see \cite[Corollary p. 440]{BS90}). On the other hand, if $R$ is a polynomial ring on a field of prime characteristic $p>0$, then $\cd(I_t(X)) = \hgt(I_t(X)) = (m-t+1)(n-t+1)$ by \cite[Proposition 4.1, p. 110]{PS73}, since $I_t(X)$ is a perfect ideal in light of \cite{HE71}.

In Sections \ref{sectionscrolls}, \ref{sectionjordan} and \ref{section2zeros}, we will see that a similar result occurs also for some classes of non-generic matrices.

The following Lemma will be employed more than once in the rest of the paper. Even if it was proved in \cite[Lemma 1.19 p. 258]{S98}, we give a more explicit proof for the sake of completeness.

\begin{lem}\label{matte}
Let $R$ be a Noetherian commutative ring and $I$ be an ideal of $R$. Consider a set of variables $y_1,\dots,y_k$ and the polynomial ring $S=R[y_1,\dots,y_k]$. Then
\[
\cd_S (I+(y_1,\dots,y_k)) = \cd_R (I)+k.
\]
\end{lem}

\begin{proof} We proceed by induction on $k \geq 1$. It suffices to prove the statement for $k=1$. For simplicity, let $y=y_1$. %Consider the graduation on $S=R[y]$ given by $\deg(y)=1$ and interpret $I$ as an ideal in $S$. Then the module $H^i_I(S)$ is also graded, for all $i \in \N$. With respect to the given graduation, $I$ is generated by degree zero elements. Clearly $H^i_I(S)_j = 0$ for every integer $i \in \N$ and $j<0$.
Consider the following long exact sequence
\[
\cdots \rightarrow H_I^c(S) \stackrel{\varphi}{\rightarrow} (H_I^c(S))_y \rightarrow H_{I+(y)}^{c+1}(S) \rightarrow H_{I}^{c+1}(S) \rightarrow \cdots.
\]
Since $S$ is a free $R$-module, it follows that $H_I^{c+1}(S)=0$. Then $H_{I+(y)}^{c+1}(S)$ is the cokernel of the map $\varphi$, and hence it is isomorphic to $H_I^c(S_y/S)$, which is nonzero since $S_y/S$ is a free $R$-module.
%
%
%and note that the maps in the sequence are homogeneous of degree $0$. Thus we can restrict the previous sequence to graded pieces. Set $c = \cd_R (I)$. Then for all $j<0$, we have the exact sequence
%\[
%\cdots \rightarrow [H_I^c(S)]_j \rightarrow [(H_I^c(S))_y]_j \rightarrow [H_{I+(y)}^{c+1}(S)]_j \rightarrow [H_{I}^{c+1}(S)]_j \rightarrow \cdots .
%\]
%
%Notice that $[H_I^c(S)]_j=[H_{I}^{c+1}(S)]_j=0$, thus the previous sequence yields
%\[
%[(H_I^c(S))_y]_j \cong [H_{I+(y)}^{c+1}(S)]_j, \text{ for each } j<0.
%\]
%
%Since $S=R[y]=\bigoplus_{h \in \N}Ry^h$ as an $R$-module and $H_I^i(-)$ commutes with direct limits, we have
%\[
%H_I^c(S)=\bigoplus_{h \in \N}H_I^c(Ry^h)=H_I^c(R)[y] \neq 0,
%\]
%where the last equality follows by $H_I^c(Ry^h)\! \cong \! H_I^c(R)$. Then $[(H_I^c(S))_y]_j \!\neq \! 0$, for $j<0$. Hence $[H_{I+(y)}^{c+1}(S)]_j \! \neq \! 0$, for each $j<0$. In particular, $H_{I+(y)}^{c+1}(S) \neq 0$,

Thus $\cd_S(I+(y)) \geq c+1$ and, on the other hand, the inequality $\cd_S(I+(y)) \leq \cd_S(I)+1$ is clear. Notice that $\cd_S(I) = c$ by virtue of the invariance of local cohomology with respect to the change of basis.
\end{proof}

\section{Kronecker-Weierstrass decomposition} \label{sectionKW}
Let $K$ be an algebraically closed field and $R$ be a polynomial ring over $K$. We require $K$ to be algebraically closed in order to transform the matrix into the Kronecker-Weierstrass form, but we can drop this assumption if the matrix is already in that form.

We consider a $(2\times n)$-matrix $M$, whose entries are linear forms of $R$. From the Kronecker-Weierstrass theory of matrix pencils, there exist two invertible matrices $C$ and $C'$ such that the matrix $X=CMC'$ is a concatenation of blocks,
\begin{equation}\label{KWform}
X=\big(N_{n_1}|\cdots|N_{n_c}|J_{\lambda_1,m_1}|\cdots|J_{\lambda_d,m_d}|B_{\ell_1}|\cdots|B_{\ell_g}\big),
\end{equation}
where the blocks are matrices of the form

$$N_{n_i}=\begin{pmatrix}
x_{i,1}&x_{i,2}& \cdots & x_{i,n_i}&0\\
0&x_{i,1}& \cdots & x_{i,n_{i-1}}&x_{i,n_i}
\end{pmatrix},$$

$$J_{\lambda_j,m_j}=\begin{pmatrix}
y_{j,1}&y_{j,2}& \cdots & y_{j,m_j}\\
\lambda_jy_{j,1}& y_{j,1}+\lambda_jy_{j,2}&\cdots&y_{j,m_{j-1}}+\lambda_jy_{j,m_j}
\end{pmatrix},$$

$$B_{\ell_p}=\begin{pmatrix}
z_{p,0}&z_{p,1}& \cdots &z_{p,\ell_{p-2}}&z_{p,\ell_{p-1}}\\
z_{p,1}&z_{p,2}& \cdots &z_{p,\ell_{p-1}}&z_{p,\ell_p}
\end{pmatrix}.$$

Here, $\mathbf{x}=\{x_{i,h}\},\mathbf{y}=\{y_{j,h}\},\mathbf{z}=\{z_{p,h}\}$ are independent linear forms of $R$, $c,d,g \geq 0$, $n_i, m_j, \ell_p$ are positive integers, and $\lambda_j \in K$. We call $N_{n_i}$ \textit{nilpotent block} of length $n_i+1$, $J_{\lambda_j,m_j}$ \textit{Jordan block} of length $m_j$ and eigenvalue $\lambda_j$ and $B_{\ell_p}$ \textit{scroll block} of length $\ell_p$, respectively. The number of scroll and nilpotent blocks $g$ and $c$, together with the lengths $\ell_p$ and $n_i$ of each of these blocks, are invariants for $M$, while the eigenvalues $\lambda_j$ of the Jordan blocks and the length $m_j$ of each of them are not invariant.  We call the matrix $X$ a \textit{Kronecker-Weierstrass normal form} of $M$. Since the matrices $C$ and $C'$ are invertible, the determinantal ideals defined by $X$ and $M$ coincide. For a detailed discussion of Kronecker-Weierstrass theory we refer to  \cite[Chapter 12]{G59}.

\begin{rem}
We point out that the blocks of length $1$ are the following:
\begin{equation*}
N_1=\begin{pmatrix}0 \\0\end{pmatrix}, \ \ \ J_{\lambda,1}=\begin{pmatrix}y_1 \\ \lambda y_1\end{pmatrix} \text{and} \ \ \ B_{1}=\begin{pmatrix}z_0 \\ z_1\end{pmatrix}.
\end{equation*}
In particular, a $(2\times n)$-matrix with generic entries is a concatenation of exactly $n$ scroll blocks of the form $B_1$.
\end{rem}

\begin{exa}
Consider the following matrix of linear forms over the polynomial ring $K[x_1,\dots,x_6]$
\begin{equation*}
\begin{pmatrix}
x_1+x_6 & x_2 & x_2+x_3 & x_4 & x_2+x_6 & x_4 \\
-x_6 & x_1 & x_1-x_3+x_4 & -x_4+x_5 & x_1-x_6 & -x_4+x_5+x_6
\end{pmatrix}.
\end{equation*}
Subtracting the second column from the fifth and the fourth from the sixth, we get
\begin{equation*}
\begin{pmatrix}
x_1+x_6 & x_2 & x_2+x_3 & x_4 & x_6 & 0 \\
-x_6 & x_1 & x_1-x_3+x_4 & -x_4+x_5 & -x_6 & x_6
\end{pmatrix}.
\end{equation*}
Subtracting the second column from the third and the fifth from the first, we get
\begin{equation*}
\begin{pmatrix}
x_1 & x_2 & x_3 & x_4 & x_6 & 0 \\
0 & x_1 & -x_3+x_4 & -x_4+x_5 & -x_6 & x_6
\end{pmatrix}.
\end{equation*}
Then adding the first row to the second one we obtain the canonical form
\begin{equation*}
\left(\begin{array}{cc|cc|cc}
x_1 & x_2 & x_3 & x_4 & x_6 & 0 \\
x_1 & x_1+x_2 & x_4 & x_5 & 0 & x_6
\end{array}\right),
\end{equation*}
which is a concatenation of a Jordan block $J_{1,2}$ of length $2$ and eigenvalue $1$, a scroll block $B_2$ of length $2$ and a nilpotent block $N_2$ of length $2$.
\end{exa}

When the matrix is in the Kronecker-Weierstrass form, a result due to Nasrollah Nejad and Zaare-Nahandi allows us to easily compute the height of the ideal of $2$-minors. Since we will use it several times, we state it here for ease of reference.

\begin{prop}{\bf (Nasrollah Nejad, Zaare-Nahandi, \cite[Proposition 2.2]{NZ11})}\label{abbasprop}
Let $X$ be a matrix in the Kronecker-Weierstrass form \eqref{KWform}. Then the height of $I_2(X)$ in $K[\mathbf{x},\mathbf{y},\mathbf{z}]$ is given by the following formulas.
\begin{enumerate}
\item If $X$ consists of exactly $c\geq1$ nilpotent blocks, then
\begin{equation*}
\mathrm{ht}\big(I_2(X)\big)=\sum_{i=1}^cn_i.
\end{equation*}
\item If $X$ consists of $c \geq 0$ nilpotent blocks and $g \geq 1$ scroll blocks, then
\begin{equation*}
\mathrm{ht}\big(I_2(X)\big)=\sum_{i=1}^cn_i+\sum_{p=1}^{g}\ell_p-1.
\end{equation*}
\item If $X$ consists of $c \geq 0$ nilpotent blocks, $g \geq 0$ scroll blocks and $d\geq1$ Jordan blocks, then
\begin{equation*}
\mathrm{ht}\big(I_2(X)\big)=\sum_{i=1}^cn_i+\sum_{p=1}^{g}\ell_p+\sum_{j=1}^dm_j-\gamma,
\end{equation*}
where $\gamma$ is the maximum number of Jordan blocks with the same eigenvalue.
\end{enumerate}
\end{prop}
We are interested in computing the cohomological dimension and the arithmetical rank of $I_2(X)$ for some special Kronecker-Weierstrass decompositions. We begin with some easy cases.

If $X$ is $(2\times2)$-matrix, then the ideal $I=I_2(X)$ is principal. Hence $\cd(I)=\ara(I)=1$, provided that $I$ is not the zero ideal. The first non trivial case occurs for matrices of size $2 \times 3$. In \cite[Corollary 6.5]{HKM09}, Huneke, Katz, and Marley proved that, if $A$ is a commutative Noetherian ring containing the field of rational numbers, with $\dim(A) \leq 5$, and $I=I_2(M)$ is the ideal generated by the $2$-minors of a $(2 \times 3)$-matrix $M$ with entries in $A$, then $H_I^{3}(A)=0$. In the following remark we show that, under these assumptions, the arithmetical rank is strictly less than $3$ whenever $M$ is a matrix of linear forms.

\begin{rem}\label{2x3case}
Let $A,M$ and $I$ be as the above. Suppose that $M$ is in the Kronecker-Weierstrass form. If $M$ contains at least one nilpotent block, the result is clear. If $M$ consists of only scroll blocks, the arithmetical rank has been settled in \cite{BV10} and the cohomological dimension is explicitly studied in Section \ref{sectionscrolls}. On the other hand, the case of a concatenation of Jordan blocks is studied in Section \ref{sectionjordan}. It remains to consider the concatenation of scroll and Jordan blocks. The matrix $M$ with a scroll block of length $2$ and a Jordan block of length $1$ is a special case of \cite[Theorem 2.1]{V81}. Suppose now that $M$ consists of two Jordan blocks of length $1$ and one scroll block of length $1$. If the Jordan blocks have the same eigenvalue, then $M$ can be transformed into a matrix with two zeros on the same row, hence $I_2(M)$ is a squarefree monomial ideal generated by $2$ monomials and the arithmetical rank is $2$. This is also the case if $M$ consists of a scroll block of length $1$ and a Jordan block of length $2$. Otherwise, if the Jordan blocks have different eigenvalues, $M$ can be transformed into a matrix with one zero and the arithmetical rank is $2$ in light of \cite[Example 2]{B07}. This is also the case if $M$ has two scroll blocks of length $1$ and a Jordan block of length $1$. Thus we completely settle the case of $(2 \times 3)$-matrices of linear forms.
\end{rem}

This is the starting point of our investigation about the cohomological dimension and the arithmetical rank of determinantal ideals of $(2 \times n)$-matrices of linear forms.

\begin{exa}\label{purecase}
Let $X$ be a $(2\times (n+1))$-matrix in the Kronecker-Weierstrass form and assume that $X$ consists of exactly one block.
\begin{itemize}
\item[i)] If $X=B_{n+1}$ is a scroll block, where
\begin{equation*}
B_{n+1}=\begin{pmatrix}
z_0 & z_1 & \cdots  & z_{n-1} & z_{n}\\
z_1 & z_2 & \cdots  & z_{n} & z_{n+1}
\end{pmatrix},
\end{equation*}
then $I_2(X)$ is the defining ideal of a rational normal curve of degree $n$ in $\mathbb{P}^n$. In \cite{RV83}, Robbiano and Valla proved that $I_2(X)$ is set-theoretic complete intersection with  $\mathrm{ht}(I_2(X))=\cd(I_2(X))=\ara(I_2(X))=n$. In particular $\sqrt{I_2(X)}=\sqrt{(F_1,\dots,F_n)}$, where
\begin{equation*}
F_i(z_0,\dots,z_{n+1})=\sum_{\alpha=0}^i(-1)^{\alpha}\binom{i}{\alpha}z_{i+1}^{i-\alpha}z_{\alpha}z_i^{\alpha}, \ \ \ \ \ \ i=1,\dots,n.
\end{equation*}

\item[ii)] If $X=N_n$ is a nilpotent block of length $n+1$, where
    \begin{equation}\label{purenilpotent}
    N_n=\begin{pmatrix}
    x_1 & x_2 & \cdots & x_n & 0\\
    0 & x_1 & \cdots & x_{n-1} & x_n
    \end{pmatrix},
    \end{equation}
    it easy to check that $\sqrt{I_2(X)}=(x_1,\dots,x_n)$. Then $\hgt(I_2(X))=\cd(I_2(X))=\ara(I_2(X))=n$. In particular, $I_2(X)$ is set-theoretic complete intersection.

\item[iii)] If $X=J_{\lambda,n+1}$ is a Jordan block of eigenvalue $\lambda$ and length $n+1$, then, by subtracting $\lambda$ times the first row from the second one, we transform the matrix into the following:
    % ORIGINAL JORDAN MATRIX BEFORE TRANSFORMATION
    % \begin{equation*}
    % X=\begin{pmatrix}
    % y_{1}&y_{2}& \cdots & y_{m}\\
    % \lambda y_{1}& y_{1}+\lambda y_{2}&\cdots& y_{m-1}+\lambda y_{m}
    % \end{pmatrix}
    % \end{equation*}
    \begin{equation*}
    \begin{pmatrix}
    y_{1}&y_{2}& \cdots &y_{n} & y_{n+1}\\
    0 & y_{1}&\cdots & y_{n-1} & y_{n}
    \end{pmatrix}.
    \end{equation*}
    It is now easy to see that $\sqrt{I_2(X)}=(y_1,\dots,y_{n})$. Then $I_2(X)$ is set-theoretic complete intersection with $\mathrm{ht}(I_2(X))=\cd(I_2(X))=\ara(I_2(X))=n$.
\end{itemize}
\end{exa}

Remark \ref{2x3case} and Example \ref{purecase} describe completely the situation where the number of blocks is $1$ or the number of columns is $n=3$, respectively. So for the rest of the paper we may assume, if necessary, that the number of blocks is at least $2$ and $n \geq 4$.

As it appears in Example \ref{purecase}, the ideal of minors of nilpotent blocks correspond to linear subspaces. These are complete intersections. Precisely we have the following result.

\begin{prop}\label{nilpotentiprop}
Let $X=(l_i)$ be a matrix of linear forms, where $l_i \in R=K[y_1,\dots,y_m]$. Let $J=I_2(X)$, $N_n$ be a nilpotent block of length $n+1$ as in \eqref{purenilpotent} and $S=R[x_1,\dots,x_n]$. Consider the matrix $M=(X|N_n)$ given by the concatenation of $X$ and $N_n$, then:
\begin{equation*}
\cd_S \big(I_2(M)\big)=\cd_R(J)+n\qquad \text{ and }\qquad  \ara \big(I_2(M)\big) \leq \ara(J)+n.
\end{equation*}
\end{prop}

\begin{proof}
Set $r=\ara(J)$. Then $\sqrt{J}=(p_1,\dots,p_r)$, for some polynomials $p_i \in R$. We define $\mathfrak{n}=(x_1,\dots,x_n)$, then $\sqrt{I_2(N_n)}=\mathfrak{n}$ by Example \ref{purecase} ii). We consider the ideals $J$, $\mathfrak{n}$ and $I_2(M)$ in the ring $S$ and we prove that
\begin{equation}\label{uguaglianzanilpotenti}
\sqrt{I_2(M)}=\sqrt{\sqrt{J}+\mathfrak{n}}.
\end{equation}
\par We have  $I_2(N_n)\subseteq I_2(M)$ and $J\subseteq I_2(M)$, hence $J+I_2(N_n)\subseteq I_2(M)$. It follows that
\begin{equation*}
\sqrt{\sqrt{J}+\mathfrak{n}}=\sqrt{\sqrt{J}+\sqrt{I_2(N_n)}}=\sqrt{J+I_2(N_n)}\subseteq\sqrt{I_2(M)},
\end{equation*}
where the second equality holds in general for every pair of ideals in a polynomial ring.
For the other inclusion, consider a $2$-minor $q$ of $M$. If $q$ involves two columns of $X$ or two columns of $N_n$, then clearly $q\in J$ or $q\in \mathfrak{n}$ respectively. Otherwise $q=l_ix_{\alpha}-l_jx_{\beta}$ or $q=-l_ix_1$ or $q=l_ix_n$. In any case it is clear that $q\in\mathfrak{n}$. This shows that $I_2(M)\subset\sqrt{J}+\mathfrak{n}$, which implies  $\sqrt{I_2(M)}\subset\sqrt{\sqrt{J}+\mathfrak{n}}$.
\par From \eqref{uguaglianzanilpotenti} and Lemma \ref{matte} we get
\begin{equation*}
\cd_S\big(I_2(M)\big) = \cd_S\big(\sqrt{I_2(M)}\big) = \cd_S\left(\sqrt{\sqrt{J}+\mathfrak{n}}\right) = \cd_S\big(\sqrt{J}+\mathfrak{n}\big) = \cd_R\big(\sqrt{J}\big)+n = \cd_R(J)+n.
\end{equation*}

Moreover the equality \eqref{uguaglianzanilpotenti} implies $\ara\big(I_2(M)\big)\leq\ara(J)+n$.
\end{proof}

We close this Section by providing explicitly an upper bound for the arithmetical rank that was implicit in \cite{B07}. Let $n,k$ be positive integers and $f_1,\dots,f_k$ be polynomials in $R=K[x_1,\dots,x_n]$. We recall that a \textit{syzygy} of $(f_1,\cdots,f_k)$ is a vector $[s_1,\cdots,s_k] \in R^k$ such that $\sum_{i=1}^k s_if_i=0$.

\begin{lem}\label{maral}
Let $k \geq 2$ be an integer and $I=(f_1,\dots,f_k)$ be a homogeneous ideal in $R=K[x_1,\dots,x_n]$. Assume that there exist a positive integer $r$ and a syzygy $[g_1,\dots,g_{k-1}] \in R^{k-1}$ of $(f_1,\dots,f_{k-1})$ such that $f_k^r \in (g_1,\dots,g_{k-1})$. Then $\ara(I) \leq k-1$.
\end{lem}

\begin{proof}
Since $f_k^r \in (g_1,\dots,g_{k-1})$, there exist $h_1,\dots,h_{k-1} \in R$ such that $f_k^r = h_1g_1 + \cdots + h_{k-1}g_{k-1}$. Let $q_i = f_kh_i+f_i$ for $1 \leq i \leq k-1$. We claim that $\sqrt{I} = \sqrt{(q_1,\dots,q_{k-1})}$. Clearly $\sqrt{(q_1,\dots,q_{k-1})} \subset \sqrt{I}$, since $(q_1,\dots,q_{k-1}) \subset I$. For the other inclusion, let $g \in \sqrt{I}$. Then there exist $r_1,\dots,r_k \in R$ such that $g^s = r_1f_1 + \cdots + r_kf_k$ for some positive integer $s$. Then
\[
g^s=\sum_{i=1}^{k-1}r_iq_i-f_k \left(\sum_{i=1}^{k-1}r_ih_i-r_k \right).
\]
We claim that $f_k^{r+1} \in (q_1,\dots,q_{k-1})$. In fact,
\[
\sum_{i=1}^{k-1}g_iq_i=\sum_{i=1}^{k-1}g_i(f_kh_i+f_i)=f_k\left(\sum_{i=1}^{k-1}g_ih_i\right)+\sum_{i=1}^{k-1}g_if_i=f_k^{r+1},
\]
where the last equality holds since $[g_1,\dots,g_{k-1}]$ is a syzygy of $(f_1,\dots,f_{k-1})$. Then
\begin{eqnarray*}
g^{s(r+1)} &=& \sum_{j=0}^r (-1)^j \binom{r+1}{j} \left( \sum_{i=1}^{k-1}r_iq_i \right)^{r+1-j} \left( \sum_{i=1}^{k-1}r_ih_i-r_k \right)^j f_k^j \\
&+& (-1)^{r+1} \left( \sum_{i=1}^{k-1}r_ih_i-r_k \right)^{r+1} f_k^{r+1} \in (q_1,\dots,q_{k-1}).
\end{eqnarray*}
Hence $g \in \sqrt{(q_1,\dots,q_{k-1})}$, as desired.
\end{proof}

Up to finding a syzygy with the required properties, we are able to decrease by one the number of generators of $I$ up to radical. We give a simple application of Lemma \ref{maral}.

\begin{exa}
Let $M=\begin{pmatrix}
0 &x_1&x_2&x_3\\
x_4&x_5&x_6&x_7\
\end{pmatrix}$ and $I=I_2(M)$ in the polynomial ring $R=K[x_1,\dots,x_7]$, where $K$ is a field of characteristic $0$. We prove that $\ara(I)=4$. By \cite[Remark 5.2]{LSW14}, we have $\cd_R(I)=4$. Then $\ara(I) \geq 4$. To prove the claim it suffices to find $4$ polynomials that generate $I$ up to radical. Recall that $[i,j]$ denotes the minor corresponding to the $i$-th and $j$-th columns of $M$. Then
\[
I=([1,2],[1,3],[2,3],[1,4],[2,4],[3,4]).
\]
Notice that $[x_2,-x_1]$ is a syzygy for $([1,2],[1,3])$ and $[2,3]= x_1x_6-x_2x_5 \in (x_2,-x_1)$. Following the proof of Lemma \ref{maral}, define $q_1=-x_5[2,3]+[1,2]$ and $q_2=-x_6[2,3]+[1,3]$. Then
\[
\sqrt{I}=\sqrt{(q_1,q_2,[1,4],[2,4],[3,4])}.
\]

By the Pl\"ucker relations (see \eqref{plucker})
\[
[1,4][2,3]-[2,4][1,3]+[3,4][1,2]=0
\]
we have that $[[2,3],-[1,3],[1,2]]$ is a syzygy for $([1,4],[2,4],[3,4])$. Notice that
\[
q_2=-x_6[2,3]-(-[1,3]) \in ([2,3],-[1,3],[1,2]).
\]

Again, following the proof of Lemma \ref{maral}, we define $p_1=-x_6q_2+[1,4], p_2=-q_2+[2,4], p_3=[3,4]$. Then $\sqrt{I}=\sqrt{q_1,p_1,p_2,p_3}$, and hence $\ara(I) \leq 4$.
\end{exa}

\section{Scroll blocks} \label{sectionscrolls}

In this section we assume that the Kronecker-Weierstrass decomposition of our matrix contains only scroll blocks. We fix an algebraically closed field $K$ and some integers $d \geq 2$ and $n_1,n_2,\dots,n_d>0$. We consider the matrix
%We consider an algebraically closed field $k$ and integers $d\geq 2$ and $n_1,n_2,\dots,n_d>0$. The $d$-dimensional rational normal scroll $S_{n_1,\dots,n_d}$ is defined as the rank one determinantal variety associated to the matrix

\begin{equation}\label{scrollmatrix}
M=(B_{n_1}|\cdots|B_{n_d})= \left(\begin{array}{cccc|c|cccc}
x_{1,0} & x_{1,1} & \dots & x_{1,n_1-1}  & \dots   & x_{d,0} & x_{d,1} & \dots & x_{d,n_d-1} \\
x_{1,1} & x_{1,2} & \dots & x_{1,n_1}  & \dots & x_{d,1} & x_{d,2} & \dots & x_{d,n_d}
\end{array}\right),
\end{equation}
where $x_{i,j}$ are algebraically independent variables over $K$. We also denote by $N=\sum_{i=1}^dn_i+d-1$ the number of variables minus $1$ and by $I_{n_1,\dots,n_d}=I_2(M)$ the homogeneous ideal generated by the $2$-minors of the matrix $M$ in the polynomial ring $R=K[x_{i,j}]$. %The ideal $I$ depends on the integers $d$ and $n_1,n_2,\dots,n_d$, but we omit any subscript for ease of notation.
\par The projective variety $R_{n_1,\dots,n_d}=\mathrm{Proj}(R/I_{n_1,\dots,n_d})\subset\mathbb{P}^N_{K}$ associated to $I_{n_1,\dots,n_d}$ has dimension $d$ and is called $d$-dimensional \emph{rational normal scroll}. These varieties have been widely studied and many properties are known. In the following Proposition we collect a few facts that will be used later on.  For a proof and a survey on rational normal scrolls the reader may consult \cite[Chapter 2]{R97}.

\begin{prop} Let $d\geq2$, $n_1,\dots,n_d>0$ be integers and let $I_{n_1,\dots,n_d}$, $R$ and $R_{n_1,\dots,n_d}$ be as above. Then
\begin{enumerate}
\item $R_{n_1,\dots,n_d}$ is irreducible, i.e. $I_{n_1,\dots,n_d}$ is a prime ideal,
\item $R/I_{n_1,\dots,n_d}$ is a Cohen-Macaulay ring of dimension $d+1$,
\item $\mathrm{Pic}(R_{n_1,\dots,n_d})\cong \mathbb{Z}\oplus\mathbb{Z}$, where $\mathrm{Pic}(R_{n_1,\dots,n_d})$ is the Picard group of $R_{n_1,\dots,n_d}$.
\end{enumerate}
\end{prop}

\par In their paper \cite{BV10}, B\u{a}descu and Valla proved that $\ara(I_{n_1,\dots,n_d})=N-2$. They exhibit $N-2$ polynomials which generate the rational normal scroll set-theoretically and they use Grothendieck-Lefschetz theory to show that $\ara(I_{n_1,\dots,n_d})\geq N-2$. In particular, it turns out that $R_{n_1,\dots,n_d}$ is a set-theoretic complete intersection if and only if $d=2$ and, in this case, $\mathrm{ht}(I_{n_1,n_2})=\cd_R(I_{n_1,n_2})=\ara(I_{n_1,n_2})=n_1+n_2-1$.
%We recall that the fact that $S_{n_1,n_2}$ is s.t.c.i in $\mathbb{P}_k^{n_1+n_2+1}$ had been proved already by  Robbiano and Valla in \cite{RV83} and Verdi in \cite{V86}.
\par The goal of this section is to compute the cohomological dimension of $I_{n_1,\dots,n_d}$. We are going to prove the following result.

\begin{thm}\label{scrolltheorem}
Let $K$ be an algebraically closed field, $d\geq2$ and $n_1,\dots,n_d>0$ integers, and $I_{n_1,\dots,n_d}=I_2(M)$ be the ideal generated by the $2$-minors of the matrix \eqref{scrollmatrix} in the polynomial ring $R=K[x_{i,j}]$ in $N+1$ variables.
Then
\begin{equation*}
\cd_R(I_{n_1,\dots,n_d})=
\begin{cases}
\mathrm{ht}(I_{n_1,\dots,n_d})=N-d=\displaystyle\sum_{i=1}^dn_i-1 &  \text{if } \chara(K)=p>0 \\
\ara(I_{n_1,\dots,n_d})=N-2=\displaystyle\sum_{i=1}^dn_i+d-3 & \text{if } \chara(K)=0
\end{cases}.
\end{equation*}
%\begin{enumerate}
%\item if $\charak=p>0$ then $\cd_S(I)=\mathrm{height}(I)=N-d$,
%\item if $\charak=0$ then $\cd_S(I)=\ara(I)=N-2$.
%\end{enumerate}
\end{thm}

The proof of this theorem will use geometric tools. In fact, we will study the variety $R_{n_1,\dots,n_d}$ rather than the ideal $I_{n_1,\dots,n_d}$. We recall some Algebraic Geometry facts. When not explicitly stated, we refer to \cite{H77} and \cite[Chapter 20]{BS98} for proofs and further details.
\par Let $S=\bigoplus_{n\in\mathbb{N}}S_n$ be a positively graded ring where $S_0$ is a field and let $\mathfrak{m}=\bigoplus_{n>0}S_n$ its homogeneous maximal ideal. We consider a finitely generated graded $S$-module $N$ and the associated coherent sheaf $\mathcal{F}=\widetilde{N}$ on $X=\mathrm{Proj}(S)$. The Serre-Grothendieck Correspondence states that there are isomorphisms of $S_0$-modules between the sheaf cohomology modules and the local cohomology modules:
\begin{equation}\label{localsheaf}
H^i(X,\mathcal{F}(n))\cong H^{i+1}_{\mathfrak{m}}(N)_n,
\end{equation}
for all $i>0$ and $n\in\mathbb{Z}$.
\par The \emph{cohomological dimension of $X$} is defined as
\begin{equation*}
\cd(X)=\min\{n\in\mathbb{N}: \ H^i(X,\mathcal{F})=0 \text{ for every } i>n \text{ and } \mathcal{F} \text{ coherent sheaf over } X\}.
\end{equation*}

If $S_0$ is a field and $I$ is a homogeneous non-nilpotent ideal, then by a result of Hartshorne \cite{H68} we have
\begin{equation}\label{localsheafdim}
\cd_S(I)-1=\cd(\mathrm{Proj}(S)\setminus\mathrm{Proj}(S/I)).
\end{equation}

Thus, in order to bound $\cd_S(I)$, we can find bounds on $\cd(X\setminus Y)$, where $Y=\mathrm{Proj}(S/I)$.
\par When the base field $S_0$ is the field of complex numbers $\mathbb{C}$, we have a strong connection between the vanishing of the sheaf cohomology modules $H^i(X\setminus Y, -) $ and the singular cohomology groups $H^i_{\text{sing}}(X_{\text{an}},\mathbb{C})$ and $H^i_{\text{sing}}(Y_{\text{an}},\mathbb{C})$. Here $X_{\text{an}}$ and $Y_{\text{an}}$ denote $X$ and $Y$ regarded as topological spaces with the euclidean topology and are called analytification of $X$ and $Y$.

\begin{thm} {\bf (Hartshorne \cite[Theorem 7.4, p. 148]{H70})}\label{hartshornetheorem}
Let $X$ be a complete scheme of dimension $N$ over $\mathbb{C}$, $Y$ be a closed subscheme, and assume that $X\setminus Y$ is non-singular. Let $r$ be an integer. Then $\cd(X\setminus Y)<r$ implies that the natural maps
\begin{equation*}
H^i_{\text{sing}}(X_{\text{an}},\mathbb{C})\longrightarrow H^i_{\text{sing}}(Y_{\text{an}},\mathbb{C})
\end{equation*}
are isomorphisms for $i<N-r$, and injective for $i=N-r$.
\end{thm}

The assumption $S_0=\mathbb{C}$ is not restrictive. In fact, the following Remark shows that we may assume it in many cases.

\begin{rem}\label{reductiontoc}
Let $K$ be a field of characteristic $0$, $R_K=K[x_1,\dots,x_n]$ the polynomial ring in $n$ variables over $K$ and $I$ an ideal of $R_K$. Since $R_K$ is Noetherian, $I$ is finitely generated, say $I=(f_1,\dots,f_m)$. The coefficients of the polynomials $f_i$ are elements of a finite extension of $\mathbb{Q}$, say $L$. We denote by $R_L=L[x_1,\dots,x_n]$ the corresponding polynomial ring. Notice that $L$ is a subfield of $K$ and a subfield of $\mathbb{C}$. We consider the ideal $I_L=I\cap R_L$, then $I=I_L R_K$ by construction. Set $R_{\mathbb{C}} = \mathbb{C}[x_1,\dots,x_n]$ and $I_{\mathbb{C}}=I_LR_{\mathbb{C}}$. We claim that
\begin{equation*}
\cd_{R_K}(I)=\cd_{R_{\mathbb{C}}}(I_{\mathbb{C}}).
\end{equation*}
Let $i$ and $j$ be integers, we look at the $j$-th graded piece of the local cohomology modules with support in $I$:
\begin{equation*}
H^i_I(R_K)_j=H^i_{I_L R_K}(R_L\otimes_L K)_j=H^i_{I_L}(R_L)_j\otimes_L K.
\end{equation*}
Since the field extension $L\subset K$ is faithfully flat, we have that $H^i_I(R_K)_j \neq 0$ if and only if $H^i_{I_L}(R_L)_j \neq 0$. In particular, $\cd_{R_K}(I)=\cd_{R_L}(I_L)$. The same argument applied to the ideals $I_L$ and $I_{\mathbb{C}}$ and to the faithfully flat field extension $L\subset\mathbb{C}$, yields $\cd_{R_L}(I_L)=\cd_{R_{\mathbb{C}}}(I_{\mathbb{C}})$, which proves the claim.
\end{rem}

We are now ready to prove Theorem \ref{scrolltheorem}.

\begin{proof}[Proof of Theorem \ref{scrolltheorem}]
For ease of notation, we set $I=I_{n_1,\dots,n_d}$ and $Y=R_{n_1,\dots,n_d}$.
\par If $K$ is a field of positive characteristic $p$, then the statement follows from \cite[Proposition 4.1, p. 110]{PS73}.

%First we assume that the field $K$ has positive characteristic $p$. By virtue of Proposition \ref{abbasprop}, $\mathrm{ht}(I)=N-d$. Since in general $\mathrm{ht}(I) \leq \cd_R(I)$, we need to show that $\cd_R(I)\leq \mathrm{ht}(I)$. Let $M$ be a finitely generated $R$-module, we write $H_I^i(M)$ as the direct limit
%\begin{equation*}
%\lim_{\substack{\longrightarrow \\ e\in\mathbb{N}}}\mathrm{Ext}^i_R\left(R/I^{[p^e]},M\right).
%\end{equation*}
%
%We consider a minimal $R$-free resolution $\mathcal{F}_{\bullet}$ of $R/I$. Since $I$ is a perfect ideal, $\pd_R(R/I)=\mathrm{ht}(I)$ so $\mathcal{F}_{\bullet}$ has length exactly $\mathrm{ht}(I)$. We apply the $e$-th iterate of the Frobenius functor to $\mathcal{F}_{\bullet}$ and we get a complex $\widetilde{\mathcal{F}_{\bullet}}$ of lenght $\leq\mathrm{ht}(I)$. By the flatness of the Frobenius functor, $\widetilde{\mathcal{F}_{\bullet}}$ is a free resolution of $R/I^{[p^e]}$. Applying $\mathrm{Hom}_R(-,M)$ to $\widetilde{\mathcal{F}_{\bullet}}$, we have
%\begin{equation*}
%\mathrm{Ext}^i_R\left(R/I^{[p^e]},M\right)=0
%\end{equation*}
%for $i>\mathrm{ht}(I)$. Hence also $H_I^i(M)=0$ for $i>\mathrm{ht}(I)$, which proves the claim.

Now let $\chara(K)=0$. In light of Remark \ref{reductiontoc}, we may assume $K=\mathbb{C}$.
We know that $\cd_R(I)\leq\ara(I)$ and $\ara(I)=N-2$, so we need to prove that $\cd_R(I)\geq N-2$.
\par  We consider the exponential sequence of sheaves over $Y_{\text{an}}$, the analytification of $Y$:
\begin{equation}\label{expshort}
0\rightarrow \underline{\mathbb{Z}}\rightarrow \mathcal{O}_{Y_{\text{an}}}\rightarrow\mathcal{O}_{Y_{\text{an}}}^{*}\rightarrow0,
\end{equation}
where $\underline{\mathbb{Z}}$ denotes the constant sheaf and the map $\mathcal{O}_{Y_{\text{an}}}\rightarrow\mathcal{O}_{Y_{\text{an}}}^{*}$ is given by $f\mapsto \exp(2\pi i f)$.
The sequence \eqref{expshort} induces a long exact sequence of sheaf cohomology modules, in particular we have
\begin{equation}\label{explong}
\cdots\rightarrow H^1(Y_{\text{an}},\mathcal{O}_{Y_{\text{an}}})\rightarrow H^1(Y_{\text{an}},\mathcal{O}^{*}_{Y_{\text{an}}})\rightarrow H^2(Y_{\text{an}},\underline{\mathbb{Z}})\rightarrow H^2(Y_{\text{an}},\mathcal{O}_{Y_{\text{an}}}) \rightarrow \cdots.
\end{equation}
By definition $H^1(Y_{\text{an}},\mathcal{O}^{*}_{Y_{\text{an}}})=\mathrm{Pic}(Y)$ and, since $\underline{\mathbb{Z}}$ is a constant sheaf, it follows that $H^2(Y_{\text{an}},\underline{\mathbb{Z}})=H^2_{\text{sing}}(Y_{\text{an}},\mathbb{Z})$.
An application of the GAGA principle and \eqref{localsheaf} yield $H^1(Y_{\text{an}},\mathcal{O}_{Y_{\text{an}}}) = H^1(Y,\mathcal{O}_{Y}) = H^2_{\mathfrak{m}}(R/I)_0$, where $\mathfrak{m}$ is the homogeneous maximal ideal of $R$. Since $R/I$ is a Cohen-Macaulay ring of dimension $d+1\geq3$ we have that $H^2_{\mathfrak{m}}(R/I)_0=0$, therefore \eqref{explong} yields the group injection
\begin{equation}\label{picequivalence}
\mathrm{Pic}(Y)\hookrightarrow H^2_{\text{sing}}(Y_{\text{an}},\mathbb{Z}).
\end{equation}

Now we assume that $\cd_R(I)<N-2$ and proceed by contradiction. From \eqref{localsheafdim} it follows that
\begin{equation*}
\cd(\mathbb{P}^N\setminus Y)=\cd_R(I)-1<N-2-1=N-3.
\end{equation*}

Theorem \ref{hartshornetheorem} with $r=N-3$ yields
\begin{equation*}
H^i_{\text{sing}}(\mathbb{P}^N_{\text{an}},\mathbb{C})\cong H^i_{\text{sing}}(Y_{\text{an}},\mathbb{C}) \text{ for } i<3,
\end{equation*}
which implies $\dim_{\mathbb{C}}H^i_{\text{sing}}(\mathbb{P}^N_{\text{an}},\mathbb{C}) = \dim_{\mathbb{C}} H^i_{\text{sing}}(Y_{\text{an}},\mathbb{C})$. By the Universal Coefficients Theorem, this is equivalent to
\begin{equation*}
\mathrm{rank}_{\mathbb{Z}}H^i_{\text{sing}}(\mathbb{P}^N_{\text{an}},\mathbb{Z})=\mathrm{rank}_{\mathbb{Z}}H^i_{\text{sing}}(Y_{\text{an}},\mathbb{Z}).
\end{equation*}

It is well known that
\begin{equation*}
H^i_{\text{sing}}(\mathbb{P}^N_{\text{an}},\mathbb{Z})=
\begin{cases} \mathbb{Z} & \text{if } i \text{ even}, 0 \leq i \leq 2N \\
0 & \text{otherwise}
\end{cases}.
\end{equation*}

In particular, $\mathrm{rank}_{\mathbb{Z}} H^i_{\text{sing}}(Y_{\text{an}},\mathbb{Z}) \leq 1$. On the other hand, $\mathrm{Pic}(Y)=\mathbb{Z}^2$, which contradicts \eqref{picequivalence}.
\end{proof}

From Theorem \ref{scrolltheorem} and Proposition \ref{nilpotentiprop} we immediately deduce

\begin{cor}
Let $K$ be an algebraically closed field of characteristic $0$, let $R$ be a polynomial ring over $K$ and $M$ be a $(2\times n)$-matrix of linear forms over $R$. If the Kronecker-Weierstrass decomposition of $M$ is
\begin{equation*}
(B_{n_1}|\cdots|B_{n_d}|N_{m_1}|\cdots|N_{m_c})
\end{equation*}
for some integers $d\geq2$, $c\geq0$, $n_1,\dots,n_d>0$ and $m_1,\dots,m_c\geq0$, then

\begin{equation*}
\cd_R\big(I_2(M)\big)=\ara\big(I_2(M)\big)=\sum_{i=1}^dn_i+\sum_{j=1}^cm_j+d-3.
\end{equation*}
\end{cor}

\section{Jordan blocks} \label{sectionjordan}

Let $K$ be a field of characteristic zero, $d \geq 1$ and $\alpha_i \geq 1$ for $i=1,\dots,d$. We consider the following $(2 \times n)$-matrix $M$ consisting of $\alpha_i$ Jordan blocks with eigenvalue $\lambda_i$ for $i=1,\dots,d$, such that $\alpha_i \geq \alpha_j$ if $j>i$:
\begin{equation}\label{jordanconcatenation}
M = \Big( J^1_{\lambda_1,m_{11}} \big| J^2_{\lambda_1,m_{12}} \big| \cdots \big| J^{\alpha_1}_{\lambda_1,m_{1\alpha_1}} \big| J^1_{\lambda_2,m_{21}} \big| J^2_{\lambda_2,m_{22}} \big| \cdots \big| J^{\alpha_2}_{\lambda_2,m_{2\alpha_2}} \big| \cdots \big| J^1_{\lambda_d,m_{d1}} \big| J^2_{\lambda_d,m_{d2}} \big| \cdots \big| J^{\alpha_d}_{\lambda_d,m_{d\alpha_d}} \Big).
\end{equation}

Here we use the following notation for the Jordan blocks, for $j=1,\dots,d$ and $i=1,\dots,\alpha_j$:
\[
J^i_{\lambda_j,m_{ji}}=
\begin{pmatrix}
y^i_{j,1} & y^i_{j,2} & \cdots & y^i_{j,m_{ji}} \\
\lambda_j y^i_{j,1} & y^i_{j,1} + \lambda_j y^i_{j,2} & \cdots & y^i_{j,m_{ji}-1} + \lambda_j y^i_{j,m_{ji}}
\end{pmatrix},
\]
where $m_{ji}$ is the length of the block.

Consider the ideal $I_2(M)$ in the polynomial ring $R=K[y^i_{j,h} : 1 \leq j \leq d, 1 \leq i \leq \alpha_j, 1 \leq h \leq m_{ji}]$. Let $\alpha = \sum_{i=1}^d \alpha_i$ be the number of blocks in $M$ and $N = \sum_{\substack{1 \leq j \leq d \\ 1 \leq i \leq \alpha_j}} m_{ji}$ be the number of variables in $R$.

The following Theorem shows that, even though the height of $I_2(M)$ depends on the maximum number of blocks with the same eigenvalue, the cohomological dimension equals the arithmetical rank of $I_2(M)$ and they are independent on how many blocks have the same eigenvalue.

\begin{thm} \label{jordantheorem}
Let $K$ be a field of characteristic zero and $M$ be a matrix of the form \eqref{jordanconcatenation}. Then
\[
\cd(I_2(M)) = \ara(I_2(M)) =
\begin{cases}
N-\alpha & \text{if } d=1\\
N-1 & \text{if } d>1
\end{cases}.
\]
\end{thm}

\begin{proof}
First we observe that
\begin{equation} \label{eqJordan}
\sqrt{I_2(M)} = J + L_M,
\end{equation}
where $J$ is the ideal generated by all the $N-\alpha$ variables $y^i_{j,h}$, for every $j=1,\dots,d$, $i=1,\dots,\alpha_j$ and $h=1,\dots,m_{ji}-1$. To describe the ideal $L_M$ first we simplify the notation: we denote the last variable $y^i_{j,m_{ji}}$ of each block by $y^i_j$. Then $L_M$ is the squarefree monomial ideal generated by the quadratics monomials of the form $y^i_j y^\ell_k$, for $j \neq k$, $1 \leq j,k \leq d$, $1 \leq i \leq \alpha_j$ and $1 \leq \ell \leq \alpha_k$. Notice that $L_M$ is an ideal in the ring $S=K[y^i_j : 1 \leq j \leq d, 1 \leq i \leq \alpha_j]$. The equality \eqref{eqJordan} holds because if we consider a minor involving at most one of the last columns of the blocks, then it is a multiple of some $y^i_{j,h} \in J$; otherwise if the minor involves the last columns of two blocks, then it is a multiple of some monomial $y^i_j y^\ell_k \in L_M$. This implies that $I_2(M) \subset J + L_M$, hence $\sqrt{I_2(M)} \subset J + L_M$, since $J+L_M$ is a radical ideal. Vice versa, first we show that $J \subset \sqrt{I_2(M)}$. We fix a block $J^i_{\lambda_j,m_{ji}}$ and we prove that $y^i_{j,h} \in \sqrt{I_2(M)}$ by induction on $h \geq 1$. For $h=1$, $\left( y^i_{j,1} \right)^2 = y^i_{j,1}(y^i_{j,1} + \lambda_j y^i_{j,2}) - \lambda_j y^i_{j,1} y^i_{j,2} \in I_2(M)$ since it is the minor corresponding to the first two columns of the block $J^i_{\lambda_j,m_{ji}}$. Suppose that $h>1$ and $y^i_{j,k} \in \sqrt{I_2(M)}$ for $1 \leq k < h$. Then $\left( y^i_{j,h} \right)^2 = \left( y^i_{j,h} \right)^2 - y^i_{j,h-1} y^i_{j,h+1} + y^i_{j,h-1} y^i_{j,h+1} \in \sqrt{I_2(M)}$, since $\left( y^i_{j,h} \right)^2 - y^i_{j,h-1} y^i_{j,h+1} \in I_2(M)$ is the minor corresponding to the columns $h$ and $h+1$ and $y^i_{j,h-1} y^i_{j,h+1} \in \sqrt{I_2(M)}$ by induction hypothesis. Now we prove that $L_M \subset \sqrt{I_2(M)}$. Notice that
\begin{eqnarray*}
(\lambda_k \!-\! \lambda_j) y^i_j y^\ell_k &\!=\!& (\lambda_k \!-\! \lambda_j) y^i_{j,m_{ji}} y^\ell_{k,m_{k\ell}} \\
&\!=\!& \begin{vmatrix}
y^i_{j,m_{ji}} & y^\ell_{k,m_{k\ell}} \\
y^i_{j,m_{ji}\!-\!1} + \lambda_j y^i_{j,m_{ji}} & y^\ell_{k,m_{k\ell}\!-\!1} + \lambda_k y^\ell_{k,m_{k\ell}}
\end{vmatrix}
\!-\! \big(y^i_{j,m_{ji}} y^\ell_{k,m_{k\ell}\!-\!1} \!-\! y^i_{j,m_{ji}\!-\!1} y^\ell_{k,m_{k\ell}} \big) \in \sqrt{I_2(M)},
\end{eqnarray*}
since $y^\ell_{k,m_{k\ell}-1}, y^i_{j,m_{ji}-1} \in J \subset \sqrt{I_2(M)}$. This yields the equality \eqref{eqJordan}.

If $d=1$, all the blocks have the same eigenvalue $\lambda_1$. Hence $L_M=(0)$ and $\sqrt{I_2(M)}=J$. This implies that $\cd(I_2(M))=\ara(I_2(M))=N-\alpha$.

Let $d \geq 2$. Notice that $L_M$ is the edge ideal of a complete $d$-partite graph $K_{\alpha_1,\alpha_2,\dots,\alpha_d}$. By \cite[Theorem 4.2.6]{J04}, we have $\cd(L_M)=\pd_S(S/L_M)=\alpha-1$. Then $\cd(I_2(M))=\cd(J)+\cd(L_M)=N-\alpha+\alpha-1=N-1$ by Proposition \ref{nilpotentiprop}.

Now we show that $\ara(I_2(M)) \leq N-1$. In light of Example \ref{purecase} iii), $\ara \big( I_2 \big(J^i_{\lambda_j,m_{ji}} \big) \big) = m_{ji}-1$ and $I_2 \big( J^i_{\lambda_j,m_{ji}} \big)$ is generated by the variables $y^i_{j,1}, y^i_{j,2}, \dots, y^i_{j,m_{ji}-1}$ up to radical.

Since $J$ is generated by $N-\alpha$ variables, in order to prove the claim, it suffices to show that $L_M$ is generated by $\alpha-1$ polynomials up to radical. We construct the following matrix with $\sum_{i=2}^d \alpha_i$ rows and $\sum_{i=1}^{d-1} \alpha_i$ columns:
\renewcommand{\arraystretch}{0.8}
\begin{small}
\begin{equation*} Q=
\left(\begin{array}{cccc|ccc|c|ccc}
y^1_1 y^1_d & y^2_1 y^1_d & \cdots & y^{\alpha_1}_1 y^1_d & y^1_2 y^1_d & \cdots & y^{\alpha_2}_2 y^1_d & \cdots & y^1_{d-1} y^1_d & \cdots & y^{\alpha_{d-1}}_{d-1} y^1_d \\
y^1_1 y^2_d & y^2_1 y^2_d & \cdots & y^{\alpha_1}_1 y^2_d & y^1_2 y^2_d & \cdots & y^{\alpha_2}_2 y^2_d & \cdots & y^1_{d-1} y^2_d & \cdots & y^{\alpha_{d-1}}_{d-1} y^2_d \\
\vdots & \vdots & \ddots & \vdots & \vdots & \ddots & \vdots & \vdots & \vdots & \ddots & \vdots \\
y^1_1 y^{\alpha_d}_d & y^2_1 y^{\alpha_d}_d & \cdots & y^{\alpha_1}_1 y^{\alpha_d}_d & y^1_2 y^{\alpha_d}_d & \cdots & y^{\alpha_2}_2 y^{\alpha_d}_d & \cdots & y^1_{d-1} y^{\alpha_d}_d & \cdots & y^{\alpha_{d-1}}_{d-1} y^{\alpha_d}_d\\
y^1_1 y^1_{d-1} & y^2_1 y^1_{d-1} & \cdots & y^{\alpha_1}_1 y^1_{d-1} & y^1_2 y^1_{d-1} & \vdots & y^{\alpha_2}_2 y^1_{d-1} & & & & \\
\vdots & \vdots & \ddots & \vdots & \vdots & \ddots & \vdots & & & & \\
y^1_1 y^{\alpha_3}_3 & y^2_1 y^{\alpha_3}_3 & \cdots & y^{\alpha_1}_1 y^{\alpha_3}_3 & y^1_2 y^{\alpha_3}_3 & \cdots & y^{\alpha_2}_2 y^{\alpha_3}_3 & & & & \\
y^1_1 y^1_2 & y^2_1 y^1_2 & \cdots & y^{\alpha_1}_1 y^1_2 & & & & & & & \\
\vdots & \vdots & \ddots & \vdots & & & & & & & \\
y^1_1 y^{\alpha_2}_2 & y^2_1 y^{\alpha_2}_2 & \cdots & y^{\alpha_1}_1 y^{\alpha_2}_2 & & & & & & & \\
\end{array}\right),
\end{equation*}
\end{small}
The first block of $Q$ is obtained by multiplying the variables $y^i_1$ by $y^h_j$ for $2 \leq j \leq d$ and $1 \leq h \leq \alpha_j$; the second block is obtained by multiplying the variables $y^i_2$ by $y^h_j$ for $3 \leq j \leq d$ and $1 \leq h \leq \alpha_j$ and so on.

Let $T$ to be the set of all the entries of $Q$, that are the generators of $L_M$. For every $\ell=1,\dots,\alpha-1$, we define $T_\ell$ as the set of all the monomials of the $\ell$-th antidiagonal of $Q$ and $q_\ell$ as the sum of these monomials. In particular, $T_1 = \{y^1_1 y^1_d\}$ and $T=\bigcup_{\ell=1}^{\alpha-1} T_\ell$. To show the last equality, we count the number of nonzero antidiagonals of $Q$. Every element on the first row is contained in exactly one $T_\ell$, hence we have $\sum_{i=1}^{d-1} \alpha_i$ sets. Moreover, every nonzero element in the last column is contained in exactly one $T_\ell$, thus we have $\alpha_d$ sets. In total we have $\sum_{i=1}^d \alpha_i - 1$ sets, since the element $y^{\alpha_{d-1}}_{d-1} y^1_d$ has been counted twice. All the other antidiagonals of $Q$ are zero because the elements of the form $y^{\alpha_j}_j y^{\alpha_h}_h$ belong to the $(\alpha-1)$-th antidiagonal. This shows that the first two conditions of Lemma \ref{SchmittVogel} are fulfilled.

As for the third condition, if we pick two monomials on the $\ell$-th antidiagonal of $Q$, they have the form $y^{i_1}_{j_1} y^{i_2}_{j_2}$ and $y^{h_1}_{k_1} y^{h_2}_{k_2}$. We may assume that either $j_1 < k_1$ or ($j_1 = k_1$ and $i_1 < h_1$). Hence their product $y^{i_1}_{j_1} y^{i_2}_{j_2} \cdot y^{h_1}_{k_1} y^{h_2}_{k_2}$ is a multiple of $y^{i_1}_{j_1} y^{h_2}_{k_2}$ that belongs to the $m$-th antidiagonal, for some $1 \leq m < \ell$ (this element is placed in the intersection of the column containing $y^{i_1}_{j_1} y^{i_2}_{j_2}$ and the row containing $y^{h_1}_{k_1} y^{h_2}_{k_2}$). From Lemma \ref{SchmittVogel} it follows that $L_M = \sqrt{L_M} = \sqrt{(q_1,\dots,q_{\alpha-1})}$ and thus $\ara(L_M) \leq \alpha-1$. Therefore
\[
\ara(I_2(M)) \leq \ara(J) + \ara(L_M) \leq N-\alpha+\alpha-1 = N-1.
\]
\end{proof}

From Theorem \ref{jordantheorem} and Proposition \ref{nilpotentiprop} we deduce

\begin{cor}
Let $K$ be a field of characteristic $0$, let $R$ be a polynomial ring over $K$ and $M'$ be a $(2\times n)$-matrix of linear forms over $R$. Suppose that the Kronecker-Weierstrass decomposition of $M'$ is
\begin{equation*}
(M|N_{m_1}|\cdots|N_{m_c})
\end{equation*}
for some integers $d\geq1$, $c\geq0$, $\alpha_1,\dots,\alpha_d \geq 1$ and $m_1,\dots,m_c\geq0$, and where $M$ is the matrix \eqref{jordanconcatenation}. Then
\begin{equation*}
\cd \big(I_2(M')\big) = \ara\big(I_2(M')\big) =
\begin{cases}
N - \alpha + \sum_{k=1}^c m_k & \text{if } d=1\\
N - 1 + \sum_{k=1}^c m_k & \text{if } d > 1
\end{cases}.
\end{equation*}
\end{cor}

\section{$(2 \times n)$-matrices with a zero diagonal} \label{section2zeros}

In Sections \ref{sectionscrolls} and \ref{sectionjordan} we analyzed the cases of concatenations of scroll blocks or Jordan blocks. In this Section we study a mixed case, in which there are both scroll and Jordan blocks. Precisely, let $n \geq 2$, $R=K[x_1,\dots,x_{2n-2}]$ and $J_n = I_2(A_n)$ be the ideal generated by the $2$-minors of the matrix
\[
A_n=\begin{pmatrix}
0 & x_1 & x_2 & \cdots &x_{n-2} & x_{n-1} \\
x_n & x_{n+1} & x_{n+2} &\cdots & x_{2n-2} & 0
\end{pmatrix}.
\]

\begin{rem} We add the first row of $A_n$ to the second one and  we apply the following linear change of variables $y_i=x_i+x_{n+i}$ for every $i=1,\dots,n-2$. We get the matrix
\begin{equation*}
A'_n=\begin{pmatrix}
0 & x_1 & x_2 & \cdots &x_{n-2} & x_{n-1} \\
x_n & y_1 & y_2 &\cdots & y_{n-2} & x_{n-1}
\end{pmatrix}
\end{equation*}
which is a Kronecker-Weierstrass form of $A_n$. In particular, $A'_n=(J_{0,1}|B_1|\cdots|B_1|J_{1,1})$ is a concatenation of a Jordan block of length $1$ and eigenvalue $0$, $n-2$ scroll blocks of length $1$ and a Jordan block of length $1$ and eigenvalue $1$. From Proposition \ref{abbasprop} it follows that $\hgt(J_n)=n-1$.
\end{rem}

\begin{Notation}\label{minorsnotation}
We label the columns of $A_n$ with the indices from $0$ to $n-1$. Recall that $[i,j]$ denotes the $2$-minor $x_ix_{n+j} - x_jx_{n+i}$ corresponding to the columns $i$ and $j$.
\end{Notation}

\begin{rem}
We recall that, if $M$ is a $(2 \times n)$-matrix of indeterminates and we label the columns with indices from $0$ to $n-1$, then the Pl\"ucker relations are the following: for every $h \in \{0,\dots,n-1\}$ and for every $0 \leq j_1 < j_2 < j_3 \leq n-1$,
\begin{equation} \label{plucker}
[h,j_1][j_2,j_3] - [h,j_2][j_1,j_3] + [h,j_3][j_1,j_2] = 0.
\end{equation}
\end{rem}

As in the case of generic matrices, we find an upper bound for the arithmetical rank of $J_n$, independent of the field.

\begin{thm} \label{ara2zeros}
Let $A_n$ the matrix above with entries in a commutative ring $R$. For every $n \geq 4$,
\[
\ara(J_n) \leq 2n-5.
\]
\end{thm}

\begin{proof}
For $n \geq 4$, the ideal $J_n$ contains both monomials and minors and it can be written in the form $J_n = J'_n + J''_n$, where
\begin{gather*}
J'_n = (x_1x_n, x_2x_n, \dots, x_{n-1}x_n, x_{n-1}x_{n+1}, \dots, x_{n-1}x_{2n-2}),\\
J''_n = (x_i x_{n+j} - x_{n+i}x_j : 1 \leq i<j \leq n-2).
\end{gather*}

In particular, the ideal $J''_n$ is the ideal of $2$-minors of the submatrix $C_n$ of $A_n$, obtained by removing the first and the last column from $A_n$. We prove that $\ara(J_n) \leq 2n-5$. To do this we will define $n-1$ polynomials containing all the monomial generators of $J_n$ and $2(n-2)-4+1=2n-7$ polynomials containing all the binomial generators of $J_n$. In total we get $3n-8$ polynomials that generate $J_n$ up to radical. Then we will reduce these polynomials to $2n-5$ by summing in a suitable way some of the polynomials in the first group to some of the polynomials in the second group.

First we define the following polynomials containing all the monomial generators of $J_n$:
\begin{align*}
q_1 &= x_{n-1} x_n, \\[2mm]
q_2 &= x_1 x_n + x_{n-1} x_{n+1}, \\[2mm]
q_3 &= x_2 x_n + x_{n-1} x_{n+2}, \\[2mm]
    &\ \ \vdots \\[2mm]
q_{n-1} &= x_{n-2} x_n + x_{n-1} x_{2n-2}.
\end{align*}

From Lemma \ref{SchmittVogel}, it follows that $J'_n = \sqrt{(q_1,\dots,q_{n-1})}$. On the other hand, by applying Theorem \ref{BrunsPoset} we get $\ara(J''_n) = 2(n-2)-4+1=2n-7$, where $J''_n=\sqrt{(p_1,\dots,p_{2n-7})}$ and $p_i$ is the sum of the minors corresponding to rank $i$ elements in the poset $\Delta(C_n)$ (see \eqref{PartialOrder} and \eqref{BrunsPolynomials}).

For $n \geq 4$, we prove that $J_n = \sqrt{K_n}$, where
\[
K_n = (p_1,\dots, p_{n-4}, q_1 + p_{n-3}, q_2 + p_{n-2}, \dots, q_{n-3} + p_{2n-7}, q_{n-2}, q_{n-1}).
\]

In other words, we consider the lowest $n-4$ levels of the poset $\Delta(C_n)$ and the corresponding polynomials $p_1,\dots, p_{n-4}$ will also be generators of $J_n$ up to radical. Then each of the remaining $n-3$ polynomials $p_{n-4+i}$ will be summed to $q_i$ for $i=1,\dots,n-3$. Finally we consider $q_{n-2}$ and $q_{n-1}$.

Let $\widetilde{J}_n = \widetilde{J}'_n + \widetilde{J}''_n$, where
\[
\widetilde{J}'_n = (q_1,\dots,q_{n-1}) \text{ and } \widetilde{J}''_n = (p_1,\dots,p_{2n-7}).
\]

Notice that $\sqrt{J'_n} = \sqrt{\widetilde{J}'_n}$ and $\sqrt{J''_n} = \sqrt{\widetilde{J}''_n}$. Then
\[
\sqrt{J_n} = \sqrt{J'_n + J''_n} = \sqrt{\sqrt{J'_n} + \sqrt{J''_n}} = \sqrt{\sqrt{\widetilde{J}'_n} + \sqrt{\widetilde{J}''_n}} = \sqrt{\widetilde{J}'_n + \widetilde{J}''_n} = \sqrt{\widetilde{J}_n},
\]
where the second and the fourth equality are true for any pair of ideals. Hence it suffices to prove that $\sqrt{\widetilde{J}_n} = \sqrt{K_n}$. Of course $K_n \subset \widetilde{J}_n$, thus $\sqrt{K_n} \subset \sqrt{\widetilde{J}_n}$.

Conversely, we show that the generators of $\widetilde{J}_n$ belong to $\sqrt{K_n}$. We know that $p_1,\dots,p_{n-4},q_{n-2},q_{n-1} \in K_n$. We need to prove that
\begin{equation} \label{claim}
q_1,\dots,q_{n-3} \in \sqrt{K_n}.
\end{equation}

It will follow that $p_{n-3},\dots,p_{2n-7} \in \sqrt{K_n}$, thus $\widetilde{J}_n \subset \sqrt{K_n}$.

With respect to the Notation \ref{minorsnotation}, the polynomials $q_i$ and $p_{n-4+i}$ can be written in the form
\begin{eqnarray*}
q_1 &=& -[0,n-1],\quad q_i = -[0,i-1] -[i-1,n-1] \text{\quad for } i=2,\dots,n-1,\\
p_{n-4+i} &=& \sum_{k=0}^{\lfloor \frac{n-3+i}{2} \rfloor -i} [i+k,n-2-k] \text{\quad for } i=1,\dots,n-3.
\end{eqnarray*}

We know that $q_{n-2},q_{n-1} \in K_n$. Let $i \in \{2,\dots,n-3\}$ and suppose that $q_j \in \sqrt{K_n}$ for every $j \in \{i+1,\dots,n-1\}$. We prove that $q_i \in \sqrt{K_n}$. Notice that
\[
q_i^2 = q_i(q_i+p_{n-4+i}) - q_ip_{n-4+i}
\]

Since $q_i+p_{n-4+i} \in K_n$, it is enough to show that $-q_ip_{n-4+i} \in \sqrt{K_n}$. By using the Notation \ref{minorsnotation}, this element can be rewritten in the form
\begin{equation} \label{qipi}
-q_ip_{n-4+i} = \sum_{k=0}^{\lfloor \frac{n-3+i}{2} \rfloor -i} \Big( [0,i-1] + [i-1,n-1] \Big) [i+k,n-2-k].
\end{equation}

Let $k \in \{0,\dots,\lfloor \frac{n-3+i}{2} \rfloor -i\}$, then the $k$-th summand of \eqref{qipi} is
\begin{gather*}
\Big( [0,i-1] + [i-1,n-1] \Big) [i+k,n-2-k] \\
= [0,i-1][i+k,n-2-k] + [i-1,n-1][i+k,n-2-k] \\[2mm]
=[0,i+k][i-1,n-2-k] - [0,n-2-k][i-1,i+k] \\
+ [i+k,n-1][i-1,n-2-k] - [n-2-k,n-1][i-1,i+k] \\[2mm]
= [i-1,n-2-k] \Big( [0,i+k]+[i+k,n-1] \Big) \\
- [i-1,i+k] \Big( [0,n-2-k]+[n-2-k,n-1] \Big) \\[2mm]
= -[i-1,n-2-k] q_{i+k+1} + [i-1,i+k] q_{n-k-1},
\end{gather*}
where the second equality follows from the Pl\"ucker relations \eqref{plucker} with respect to the indices $h=0, j_1=i-1, j_2=i+k, j_3=n-2-k$ for the first summand and $h=n-1, j_1=i-1, j_2=i+k, j_3=n-2-k$ for the second summand. Hence
\[
-q_ip_{n-4+i} = \sum_{k=0}^{\lfloor \frac{n-3+i}{2} \rfloor -i} \Big( -[i-1,n-2-k] q_{i+k+1} + [i-1,i+k] q_{n-k-1} \Big),
\]
where $q_{i+k+1}, q_{n-k-1} \in \sqrt{K_n}$ since $i+k+1$ and $n-k-1$ are both greater than $i$ and less than or equal to $n-1$. Thus $q_i^2 \in \sqrt{K_n}$ and therefore $q_i, p_{n-4+i} \in \sqrt{K_n}$.

It remains to prove that $q_1 \in \sqrt{K_n}$. Notice that
\[
q_1^2 = q_1(q_1+p_{n-3}) - q_1p_{n-3}
\]

Since $q_1+p_{n-3} \in K_n$, it is enough to show that $-q_1p_{n-3} \in \sqrt{K_n}$. By using the Notation \ref{minorsnotation}, this element can be rewritten in the form
\begin{equation} \label{q1}
-q_1p_{n-3} = \sum_{k=0}^{\lfloor \frac{n-2}{2} \rfloor -1} [0,n-1] [1+k,n-2-k].
\end{equation}

Let $k \in \{0,\dots,\lfloor \frac{n-2}{2} \rfloor -1\}$, then the $k$-th summand of \eqref{q1} is
\begin{gather*}
[0,n-1][1+k,n-2-k] = [0,1+k][n-2-k,n-1] - [0,n-2-k][1+k,n-1]\\[2mm]
= [0,1+k][n-2-k,n-1] - [0,n-2-k][1+k,n-1] \\
+ [1+k,n-1][n-2-k,n-1] - [1+k,n-1][n-2-k,n-1]\\[2mm]
= [n-2-k,n-1] \Big([0,1+k] + [1+k,n-1] \Big) - [1+k,n-1] \Big( [0,n-2-k] + [n-2-k,n-1] \Big) \\[2mm]
= -[n-2-k,n-1] q_{2+k} + [1+k,n-1] q_{n-k-1},
\end{gather*}
where the first equality follows from the Pl\"ucker relations \eqref{plucker} with respect to the indices $h=0, j_1=1+k, j_2=n-2-k, j_3=n-1$. Hence
\[
-q_1p_{n-3} = \sum_{k=0}^{\lfloor \frac{n-2}{2} \rfloor -1} \Big( -[n-2-k,n-1] q_{2+k} + [1+k,n-1] q_{n-k-1} \Big),
\]
where $q_{2+k}, q_{n-k-1} \in \sqrt{K_n}$ since $2+k$ and $n-k-1$ are both greater than $1$ and less than or equal to $n-1$. Thus $q_1^2 \in \sqrt{K_n}$ and therefore $q_1, p_{n-3} \in \sqrt{K_n}$.
\end{proof}

Now we compute the cohomological dimension of $J_n$. Again, as for the generic matrices, it depends on the characteristic of the field.

\begin{thm} \label{AraCd2zeros}
Let $n \geq 2$, $R=K[x_1,\dots,x_{2n-2}]$ and $J_n = I_2(A_n)$ be the ideal generated by the $2$-minors of $A_n$. Then
\begin{itemize}
\item[i)] $\hgt(J_2)=\cd(J_2)=\ara(J_2)=1$ and $\hgt(J_3)=\cd(J_3)=\ara(J_3)=2$,
\item[ii)] for $n \geq 4$,
\[
\cd(J_n) =
\begin{cases}
\hgt(J_n)=n-1 & \text{if } \chara(K)=p>0\\
\ara(J_n)=2n-5 & \text{if } \chara(K)=0
\end{cases}.
\]
\end{itemize}
\end{thm}

First we consider the case $\chara(K)=0$. Under this assumption, not only we prove that $\cd(J_n) = 2n-5$, but we also show the vanishing of all local cohomology modules with indices between $\hgt(J_n)$ and $2n-5$.

%in two different ways. The first proof is based on simple tools, like the Mayer-Vietoris long exact sequence. On the other hand, in the second proof we use more sophisticated results about local cohomology. In this way we show the vanishing of all local cohomology modules with indices between $\hgt(J_n)$ and $2n-5$.

\begin{thm}\label{cd2zerosbis}
Let $K$ be a field of characteristic $0$ and $R=K[x_1,\dots,x_{2n-2}]$. For every $n \geq 4$,
\begin{equation*}
H^{i}_{J_n}(R)\neq 0 \Longleftrightarrow i=n-1 \text{ or }\; i=2n-5.
\end{equation*}
In particular, $\cd_R(J_n)=2n-5$.
\end{thm}

\begin{Notation}\label{notationcd2zeros} We fix $S=K[x,y,x_1,\dots,x_{2n-2}]$ and $A=S/(x)$, then  $R=A/(y)=S/(x,y)$. We consider the generic $(2 \times n)$-matrix $M_n$  over $S$
\begin{equation*}
M_n=\begin{pmatrix}
x & x_1 & \cdots & x_{n-2} & x_{n-1} \\
x_n & x_{n+1} & \cdots & x_{2n-2} & y
\end{pmatrix}
\end{equation*}
and $I=I_2(M_n)$ the ideal of $2$-minors of $M_n$. Notice that, if $x=y=0$, the ideal $I$ coincides with the ideal $J_n$. In other words, $J_n = IR$.
\end{Notation}

The basic idea is to reduce the vanishing of $H_{IR}^i(R)$ to the vanishing of $H_{I}^i(S)$ by using the multiplication maps by $x$ and by $y$. The modules $H_{I}^i(S)$ are well-understood thanks to the following results due to Witt and Lyubeznik, Singh and Walther.

\begin{thm}\bf (Witt, \cite[Theorem 1.1]{W12}) \it \label{witt}
Let $S$ and $I$ be as above. Then
\begin{equation*}
H_I^i(S) \neq 0 \Longleftrightarrow i=n-1 \text{ or } i=2n-3.
\end{equation*}
\end{thm}

\begin{thm}\bf (Lyubeznik, Singh, Walther,  \cite[Theorem 1.2]{LSW14})\it \label{LSW}
Let $S$ and $I$ be as above, and let\break $\mathfrak{m}=(x,y,x_1,\dots,x_{2n-2})$ the homogeneous maximal ideal of $S$. Then we have an isomorphism of $S$-modules
\begin{equation*}
H^{2n-3}_{I}(S)\cong H^{2n}_{\mathfrak{m}}(S).
\end{equation*}
\end{thm}

\begin{proof}[Proof of Theorem \ref{cd2zerosbis}] It is clear that $H^{n-1}_{J_n}(R)\neq0$ and $H^{i}_{J_n}(R)=0$ for $i<n-1$ , since $\hgt(J_n)=n-1$. By Theorem \ref{ara2zeros}, we have also $H_{J_n}^{i}(R)=0$ for $i>2n-5$. For $n=4$, one has that $2n-5=3=\hgt(J_4)$, then $H^3_{J_4}(R) \neq 0$.
\par Now let $n \geq 5$ and let $S$, $A$ and $I$ be as in Notation \ref{notationcd2zeros}.

\par We consider the map $S \xrightarrow{\cdot x} S$, it induces a long exact sequence of local cohomology modules:
\begin{equation}\label{seqcd2zeros1}
\cdots\rightarrow H_{I}^j(S)\rightarrow H_{IA}^j(A)\rightarrow H_{I}^{j+1}(S)\xrightarrow{\cdot x}H_{I}^{j+1}(S)\rightarrow\cdots .
\end{equation}

For $j=2n-4$ we get
\begin{equation*}
 H_{IA}^{2n-4}(A)=\mathrm{ker}\big(H_{I}^{2n-3}(S)\xrightarrow{\cdot x}H_{I}^{2n-3}(S)\big),
\end{equation*}
since $H^{2n-4}_{I}(S)=0$ by Theorem \ref{witt}. On the other hand, by using the \v{C}ech complex, it is easy to see that
\begin{equation*}
 H_{\mathfrak{m} A}^{2n-1}(A)=\mathrm{ker}\big(H_{\mathfrak{m}}^{2n}(S)\xrightarrow{\cdot x}H_{\mathfrak{m}}^{2n}(S)\big).
\end{equation*}

Then the isomorphism of Theorem \ref{LSW} yields $H_{IA}^{2n-4}(A)\cong H_{\mathfrak{m}A}^{2n-1}(A)$, the latter being non-zero since $\mathfrak{m}A$ is the homogeneous maximal ideal of $A$. Moreover, by Theorem \ref{witt}, if $n-1<j<2n-4$, then $H_{I}^j(S)=H_{I}^{j+1}(S)=0$. Therefore $ H_{IA}^j(A)=0$ by virtue of \eqref{seqcd2zeros1}. %Notice that the condition on $j$ makes sense since $n\geq 5$.

\par Now we consider the multiplication map $A \xrightarrow{\cdot y} A$ and the corresponding long exact sequence:
\begin{equation}\label{seqcd2zeros2}
\cdots\rightarrow H_{IA}^i(A)\rightarrow H_{IR}^i(R)\rightarrow H_{IA}^{i+1}(A)\xrightarrow{\cdot y}H_{IA}^{i+1}(A)\rightarrow \cdots.
\end{equation}

For $i=2n-5$, from Theorem \ref{LSW} it follows that
\begin{equation*}
H_{IR}^{2n-5}(R)=\mathrm{ker}\big(H_{IA}^{2n-4}(A)\xrightarrow{\cdot y}H_{IA}^{2n-4}(A)\big)\cong\mathrm{ker} \big(H_{\mathfrak{m}A}^{2n-1}(A)\xrightarrow{\cdot y}H_{\mathfrak{m}A}^{2n-1} (A) \big)=H_{\mathfrak{m}R}^{2n-2}(R),
\end{equation*}
since $H^{2n-5}_{IA}(A)=0$ and $H_{\mathfrak{m}A}^{2n-2}(A)=0$. Then $\mathfrak{m}R$ is the homogeneous maximal ideal of $R$, hence $H_{\mathfrak{m}R}^{2n-2}(R)\neq0$. This implies $H^{2n-5}_{J_n}(R) \neq 0$, since $J_n=IR$.
\par It remains to prove that $H^i_{J_n}(R)=0$ for $n-1<i<2n-5$. For such $i$, we have $H_{IA}^i(A)= H_{IA}^{i+1}(A)=0$, as shown above. Then \eqref{seqcd2zeros2} yields $H_{IR}^i(R)=0$, as required.
\end{proof}

\begin{proof}[Proof of Theorem \ref{AraCd2zeros}]
For $n=2$, the ideal $J_2$ is principal, thus $\cd(J_2) = \ara(J_2) = 1$. For $n=3$, we have $\cd(J_3)= \ara(J_3)=2$, as computed in Example \ref{2x3case}. Let $n \geq 4$. If $\chara(K)=0$, the claim follows from Theorem \ref{cd2zerosbis}. If $\chara(K)=p>0$, the claim follows from \cite[Proposition 4.1, p. 110]{PS73}, since $J_n$ is a perfect ideal. In fact, $\hgt(J_n)=\mathrm{grade}(J_n)=n-1$. Moreover, by \cite[Theorem 2, p. 201]{EN62}, $\pd_R(R/I)=n-2+1=n-1$.
\end{proof}

\section*{Acknowledgments}
The authors wish to thank Aldo Conca, Srikanth Iyengar and, in particular, Anurag Singh, together with the organizers of Pragmatic 2014 (Catania, Italy), for the beautiful school, the interesting lectures, many useful suggestions and insights. The authors are also in debt with Matteo Varbaro and Margherita Barile for helpful discussions and hints.

\end{document}